\documentclass[12pt,a4paper]{amsart}

\usepackage[latin1]{inputenc}
\usepackage{amsmath,amsfonts,amssymb,setspace,amscd}
\usepackage{enumerate}
\usepackage[
      colorlinks=true,    
      urlcolor=blue,    
      menucolor=blue,    
      linkcolor=blue,    
      bookmarks=true,    
      citecolor=blue,
      bookmarksopen=true,    
      hyperfootnotes=false,    
      pdfpagemode=UseOutlines    
]{hyperref}
\usepackage[margin=2.5cm]{geometry}
\onehalfspace

\usepackage{marginnote}

\newtheorem{theorem}{Theorem}[section]
\newtheorem{lemma}[theorem]{Lemma}
\newtheorem{proposition}[theorem]{Proposition}

\theoremstyle{definition}
\newtheorem*{question}{Question}
\newtheorem{definition}[theorem]{Definition}

\numberwithin{equation}{section}
\numberwithin{theorem}{section}

\newcommand{\R}{\mathbb{R}}
\newcommand{\Z}{\mathbb{Z}}
\newcommand{\Q}{\mathbb{Q}}
\newcommand{\C}{\mathbb{C}}

\newcommand{\Qbar}{\overline{\Q}}

\newcommand{\cE}{\mathcal{E}}
\newcommand{\cA}{\mathcal{A}}
\newcommand{\cC}{\mathcal{C}}

\newcommand{\bo}[1]{\boldsymbol{#1}}
\newcommand{\mc}[1]{\mathcal{#1}}

\def\lg{\left\lbrace}
\def\rg{\right\rbrace}

\newcommand{\Oseen}{\mathcal{O}}

\author[F. Barroero]{Fabrizio Barroero}

\address{Department of Mathematics and Computer Science, University of Basel, Spiegelgasse 1, 4051 Basel, Switzerland}
\email{fbarroero@gmail.com}

\title[CM relations]{CM relations in fibered powers of elliptic families}
\date{\today}


\thanks{F. B. was supported by the EPSRC grant EP/N007956/1' and the SNF grant 165525.}

\begin{document}

\begin{abstract}
Let $E_\lambda$ be the Legendre family of elliptic curves. Given $n$ linearly independent points $P_1,\dots , P_n \in E_\lambda\left(\overline{\Q(\lambda)}\right)$ we prove that there are at most  finitely many complex numbers $\lambda_0$ such that $E_{\lambda_0} $ has complex multiplication and $P_1(\lambda_0), \dots ,P_n(\lambda_0)$ are dependent over End$(E_{\lambda_0})$. This implies a positive answer to a question of Bertrand and, combined with a previous work in collaboration with Capuano, proves the Zilber-Pink conjecture for a curve in a fibered power of an elliptic scheme when everything is defined over $\Qbar$.
\end{abstract}

\maketitle

\section{Introduction}

Let $E_\lambda$ denote the elliptic curve with equation
\begin{equation}\label{defleg}
Y^2Z=X(X-Z)(X-\lambda Z).
\end{equation}
We see this as a family of elliptic curves $E_\lambda\rightarrow S=\mathbb{A}^1 \setminus \{0,1\}$ and consider the $n$-fold fibered power 
$E_\lambda\times_S \dots \times_S E_\lambda$, for some positive integer $n$. By abuse of notation we indicate this fibered power by $E_\lambda^n$. This defines again a family $E_\lambda^n\rightarrow S$.

Now, suppose we are given an irreducible curve $\mc{C}\subseteq E_\lambda^n$, defined over the algebraic numbers and not contained in a single fiber of the family. This defines $n$ points $P_1, \dots, P_n$ on $E_\lambda(\C(\cC))$, which we suppose to be independent, i.e., there is no generic non-trivial relation between them. In other words, $\cC$ is not contained in a proper subgroup scheme of $E_\lambda^n \rightarrow S$.

For any point $\bo{c} \in \mc{C}(\C)$ we have $n$ specialized points $P_1(\bo{c}), \dots, P_n(\bo{c})$ on the specialized curve $E_{\lambda (\bo{c})}$, which might become dependent over $\Z$ or over a possibly larger endomorphism ring.

In joint work with Capuano \cite{linrel} we proved that there are at most finitely many $\bo{c} \in \mc{C}(\C)$ such that $P_1(\bo{c}), \dots, P_n(\bo{c})$ satisfy two independent relations over $\Z$ (see \cite{MZ12} for the case $n=2$). The Zilber-Pink Conjecture predicts that finiteness holds as well when one considers relations over the endomorphism rings of the fibers and, in the case of CM fibers, one relation is enough.

The main result of this article is the following theorem.

\begin{theorem}\label{mainthm}
Let $\mc{C}\subseteq E_\lambda^n$ be an irreducible curve defined over $\Qbar$, not contained in a fixed fiber of $E_\lambda^n$ and such that the $n$ points $P_1, \dots , P_n$ defined by it are generically independent. Then there are at most finitely many $\bo{c} \in \mc{C}(\C)$ such that $E_{\lambda (\bo{c})}$ has complex multiplication and there exists $(a_1, \dots, a_n) \in \text{\emph{End}} \left(E_{\lambda (\bo{c})} \right)^n \setminus \{0\}$ with
$$
a_1 P_1(\bo{c}) + \dots + a_n P_n (\bo{c})=O.
$$
\end{theorem}

In case $n=1$ we have one non-identically torsion point and the theorem says that there are only finitely many specializations such that we have complex multiplication and the point has finite order. This was proved by Andr\'e \cite{andrenote} in unpublished notes and later by Pila in \cite{PilaIMRN}.


The basic idea of Andr\'e's proof is the following: by a theorem of Silverman \cite{Sil83} the height of the points $\bo{c}$ such that $P_1(\bo{c})$ is torsion is bounded while if there were infinitely many $\bo{c}$ such that $E_{\lambda (\bo{c})}$ has CM a result of Colmez \cite{Colmez} would imply that their height must tend to infinity.

In our case Silverman's theorem cannot be applied because it does not provide bounded height for the $\bo{c}$ such that the coefficients of the relation between the $P_i(\bo{c})$ are not all in $\Z$.


Pila did not use Silverman's Theorem and followed the general strategy, first introduced by Pila and Zannier in \cite{PilaZannier}, which has been very successful in proving several new instances of the Zilber-Pink conjecture. We use the same strategy and give here a sketch of the proof.
The elliptic curve $E_\lambda$ is analytically isomorphic to $\C / \Lambda_\tau$, where $\Lambda_\tau=\Z + \tau \Z$, for some $\tau$ in the complex upper-half plane. 
Let $\cC'$ be the subset of $\cC$ we want to prove to be finite. Fix $\bo{c}_0$ in $\cC'$, let $D_0$ be the discriminant of the endomorphism ring of $E_{\lambda(\bo{c}_0)}$ and let $a_1, \dots , a_n$ be the coefficients of a non-trivial relation between the $P_i(\bo{c}_0)$. By the theory of complex multiplication we have that the corresponding $\tau_0$ is imaginary quadratic and has height $\ll |D_0|$, while using works of Masser and David we can suppose that the $a_i$ have height bounded by a positive power of $|D_0|$, up to a constant. Moreover, all conjugates of $\bo{c}_0$ are in $\cC'$ with the same CM discriminant and coefficients of the relation between the $P_i$. Again, the theory of complex multiplication tells us that there are at least $|D_0|^{1/3}$ such conjugates.
We consider the elliptic logarithms $z_1, \dots, z_n$ of $P_1, \dots ,P_n$ and the uniformization map $(\tau, z_1, \dots , z_n) \mapsto (\lambda, P_1, \dots ,P_n)$.  This map, restricted to a suitably chosen fundamental domain, is definable in the o-minimal structure $\R_{\text{an,exp}}$ by a work of Peterzil and Starchenko. The preimage of $\cC$ via this map is then a definable surface. Our point $\bo{c}_0$ and all its conjugates will correspond to points on this surface lying in a linear variety defined by equations whose coefficients are related to $\tau_0$ and the $a_i$ and so are forced to be quadratic numbers of height $\ll |D_0|^\gamma$, for some positive $\gamma$. A Theorem of Habegger and Pila implies that there are at most $\ll_\epsilon |D_0|^{\gamma\epsilon}$ points of that kind on the surface, provided the functions $z_1,\dots ,z_n$ are  algebraically independent over $\C (\tau)$. This is ensured by a result of Bertrand. Finally, recalling that we have $\geq |D_0|^{1/3}$ of such points coming from conjugating $\bo{c}_0$,if we choose a small enough $\epsilon$ we have a bound on $|D_0|$ and the claim of the Theorem.

Let us see an example. Let 
$$
 P_1=\left( 2, \sqrt{2(2-\lambda)} \right), \quad P_2=\left( 3, \sqrt{6(3-\lambda)} \right).
$$
These are generically independent points on $E_\lambda$. Indeed, they are defined over disjoint quadratic extensions of $\Qbar (\lambda)$ and therefore if they were dependent, by conjugating, we see that they would be identically torsion. This is not the case. For instance, $P_1(6)$ has infinite order on $E_6$. In \cite{MasserZannier08}, Masser and Zannier proved that there are at most finitely many values $\lambda_0$ such that $P_1(\lambda_0)$ and $P_2(\lambda_0)$ are both torsion. Our theorem implies that there are at most finitely many $\lambda_0$ such that $E_{\lambda_0}$ has complex multiplication and the points $P_1(\lambda_0)$ and $P_2(\lambda_0)$ are dependent over End$(E_{\lambda_0})$. 

As mentioned above, our Theorem \ref{mainthm} is a special case of the so-called Zilber-Pink Conjectures on Unlikely Intersections. In particular, combined with results in \cite{linrel}, \cite{Viada2008} and \cite{galateau2010}, it settles the conjecture for a curve in a fibered power of an elliptic scheme, when everything is defined over $\Qbar$. For an account on these conjectures see \cite{pink}, \cite{Zannier} or \cite{HabPila14}.

Let $\mc{E}\rightarrow S$ be a non-isotrivial elliptic scheme over an irreducible, smooth, quasi-projective curve $S$, both defined over $\Qbar$. Moreover, let $\mc{A}\rightarrow S$ be its $n$-fold fibered power. An irreducible subvariety of $\mc{A}$ is called special if it is an irreducible component of an algebraic subgroup of a CM fiber or an irreducible component of a flat subgroup scheme of $\mc{A}$. We will define flat subgroup schemes in the next section, where we will also see how Theorem \ref{mainthm} implies the following statement.

\begin{theorem}\label{ZP}
Let $\mathcal{A}$ be as above and let $\cC$ be an irreducible curve in $\mathcal{A}$ defined over $\overline{\mathbb{Q}}$ and not contained in a proper special subvariety of $\mathcal{A}$. Then there are at most finitely many points in $\mathcal{C}$ that are contained in special subvarieties of $\mathcal{A}$ of codimension at least 2.
\end{theorem}

Bertrand in \cite{BertPoinc} asked the following question, as one of the ingredients needed for proving the Zilber-Pink Conjecture for curves in Poincar\'e biextensions of elliptic schemes.

\begin{question}[\cite{BertPoinc}, Question 1]
Let $\mc{E}$ be a non-isotrivial elliptic scheme over a curve $S/\Qbar$, and let $p$, $q$ be two sections of $E/S$ defined over $\Qbar$. Assume that there are infinitely many points $\overline{\lambda} \in S(\Qbar)$ such that the fiber $\mc{E}_{\overline{\lambda}}$ of $\mc{E}/S$ above $\overline{\lambda}$ admits complex multiplication and such that the points $p(\overline{\lambda})$ and $q(\overline{\lambda})$ are linearly dependent over End$\left(\mc{E}_{\overline{\lambda}}\right)$. Must the section $p$ and $q$ then be linearly dependent over $\Z$?
\end{question}

The example above is clearly an instance of such problem. In the next section we will see how Theorem \ref{ZP} implies a positive answer to Bertrand's question.

In this paper we use the $\ll$ notation: we say $a\ll b$ for non-negative real numbers $a$ and $b$ if there exists a positive $c$ such that $a\leq cb$. The constant $c$ will usually depend on $\cC$. Any further dependence will be specified with an index.

\section{Proof of Theorem \ref{ZP}}

Recall that we have a non-isotrivial elliptic scheme $\mc{E}\rightarrow S$ over an irreducible, smooth, quasi-projective curve $S$, both defined over $\Qbar$. Non-isotrivial means that $\mc{E}$ cannot become a constant family after a finite  base change. We have its $n$-fold fibered power $\mc{A}$ and an irreducible curve $\cC$ defined over $\Qbar$ and not contained in a flat subgroup scheme or in a fixed fiber. We call $\varphi$ the structural morphism $\mc{A}\rightarrow S$.

The following definitions and results are borrowed from a work of Habegger \cite{Hab}.

First, a subgroup scheme $G$ of $\mc{A}$ is a closed, possibly reducible, subvariety of $\mc{A}$ which contains the image of $G\times_S G$ under the addition morphism and the image of the zero section, and is mapped to itself by the inversion morphism. A subgroup scheme $G$ is called flat if $\varphi_G:G\rightarrow S$ is flat, i.e., all irreducible components of $G$ dominate the base curve $S$ (see \cite{Hart}, Chapter III, Proposition 9.7).

For every $\bo{a}=(a_1, \dots , a_n)\in \Z^n $ we have a morphism $\bo{a}:\cA \rightarrow \cE$ defined by
$$
\bo{a}(P_1,\dots, P_n)=a_1 P_1 +\dots +a_n P_n.
$$
We identify the elements of $\Z^n$ with the morphisms they define. The fibered product $\alpha =\bo{a}_1 \times_S \dots \times_S \bo{a}_r$, for $\bo{a}_1 , \dots , \bo{a}_r \in \Z^n$ defines a morphism $\cA \rightarrow \mathcal{B}$ over $S$ where $\mathcal{B}$ is the $r$-fold fibered power of $\cE$. The kernel of $\alpha$, $\text{ker} \, \alpha$ indicates the fibered product of $\alpha: \cA \rightarrow \mathcal{B}$ with the zero section $S\rightarrow \mathcal{B}$. We consider it as a closed subscheme of $\cA$.

\begin{lemma}\label{lemrel}
Let $G$ be a  flat subgroup scheme of $\cA$ of codimension $\geq r$, with $1\leq r\leq n$. Then, there exist independent $\bo{a}_1,\dots , \bo{a}_r \in \Z^n$ such that $G\subseteq \textnormal{ker}(\bo{a}_1 \times_S \dots \times_S \bo{a}_r)$ and $\textnormal{ker}(\bo{a}_1 \times_S\dots \times_S \bo{a}_r)$ is a flat subgroup scheme of $\cA$ of codimension $r$. Moreover, for every $s\in S(\C)$ we have $\dim G_s=\dim G-1$.
Finally, the point $(P_1,\dots, P_n)\in \mc{A}_s$ is contained in a proper algebraic subgroup of $\mc{A}_s$ if and only if there exists $(a_1, \dots , a_n)\in \text{\emph{End}}(\mc{E}_s)^n\setminus\{0\}$  with $a_1P_1 +\dots +a_n P_n=0$.
\end{lemma}
\begin{proof}
This follows from Lemma 2.5 of \cite{Hab} and its proof. The last claim is classical.
\end{proof}

By this lemma it is then clear that Theorem \ref{ZP} implies a positive answer to Bertrand question. Indeed, the two sections $p$ and $q$ correspond to a curve in the fibered square of a non-isotrivial elliptic scheme. If this curve intersects the union of special subvarieties of codimension at least 2 in infinitely many points then it must be contained in a proper special subvariety which can only be a flat subgroup scheme because the curve is not contained in a fixed fiber. Therefore, the two sections $p$ and $q$ are dependent over $\Z$.

Consider now the Legendre family with equation \eqref{defleg}. This gives an example of an elliptic scheme, which we call $\cE_L$, over the modular curve $Y(2)=\mathbb{P}^1\setminus \{0,1,\infty\}$. We write $\cA_L$ for the $n$-fold fibered power of $\cE_L$.

\begin{lemma}[\cite{Hab}, Lemma 5.4] \label{lemdiag}
Let $\cA$ be as above. After possibly replacing $S$ by a Zariski open non-empty subset there exists an irreducible non-singular quasi-projective curve $S'$ defined over $\Qbar$ such that the following is a commutative diagram
$$
\begin{CD}
\cA @<{f}<< \cA' @>\text{e}>>  \cA_L \\
@VVV @VVV @VVV\\
S @<<{l}<  S' @>>{\lambda}> Y(2)
\end{CD}
$$
where $l$ is finite, $\lambda$ is quasi-finite, $\cA'$ is the abelian scheme $\cA\times_S S'$, $f$ is finite and flat and $e$ is quasi-finite and flat. Moreover, the restriction of $f$ and $e$ to any fiber of $\cA'\rightarrow S'$ is an isomorphism of abelian varieties.
\end{lemma}

We will also need the following technical lemma.

\begin{lemma} \label{lemhab}
If $G$ is a flat subgroup scheme of $\cA$ then $e\left(f^{-1}(G)\right)$ is a flat subgroup scheme of $\cA_L$ of the same dimension.
Moreover, let $X$ be a subvariety of $\cA$ dominating $S$ and not contained in a proper flat subgroup scheme of $\cA$, $X''$ an irreducible component of $f^{-1}(X)$ and $X'$ the Zariski closure of $e(X'')$ in $\cA_L$. Then $X'$ has the same dimension of $X$, dominates $Y(2)$ and is not contained in a proper flat subgroup scheme of $\cA_L$.
\end{lemma}

\begin{proof}
This follows from the proof of Lemma 5.5 of \cite{Hab}.
\end{proof}

We can now see how Theorem \ref{ZP} follows from our Theorem \ref{mainthm} in combination with works of Viada, Galateau, and a previous work of the author with Capuano. 

First, by hypothesis $\cC$ is not contained in a fixed CM fiber. Moreover, the claim of the theorem follows from works of Viada \cite{Viada2008} and Galateau \cite{galateau2010} if $\cC$ is contained in a fixed non-CM fiber because, by Lemma \ref{lemrel}, flat subgroup schemes specialize to algebraic subgroups of the same codimension. Therefore, we can suppose $\cC$ is not contained in a fixed fiber.

Now, Theorem 2.1 of \cite{linrel} implies that $\cC$ intersects the union of flat subgroup schemes of codimension at least 2 in finitely many points. We then only have to prove that the intersection of $\cC$ with the union of all proper algebraic subgroups of CM fibers is finite.

Suppose this is not the case and consider the diagram in Lemma \ref{lemdiag}. Then, using Lemma \ref{lemdiag} and \ref{lemhab} one can see that the Zariski closure $\cC'$ of the image via $e$ of some irreducible component of $f^{-1}(\cC)$ would be a curve in $\cA_L$, which is not contained in a flat subgroup scheme or in a fixed fiber. Since the restriction of $f$ and $e$ to any fiber of $\cA'\rightarrow S'$ is an isomorphism and $l$ is a finite map, we have that $\mc{C}'$ would contain infinitely many point lying in proper algebraic subgroups of CM fibers. Therefore, we are reduced to proving the claim for the Legendre family. This follows from Theorem \ref{mainthm} by the last claim of Lemma \ref{lemrel}.

\section{Preliminaries}

In this section we introduce some notations and collect results needed later.

\subsection{Heights}
 
By $h$ we will indicate the logarithmic absolute Weil height on the projective space $\mathbb{P}^N$, as defined in \cite{BombGub}, p.~16, while $\widehat{h}$ denotes the N\'eron-Tate or canonical height defined for the algebraic points of an elliptic curve defined over the algebraic numbers. For this see \cite{Silv09}, VIII.9.

Now, if $\alpha$ is an algebraic number, we set $h(\alpha)=h([1,\alpha])$ and define its multiplicative height as $H(\alpha)=\exp(h(\alpha))$.

We need a further definition of height.
Set $\min \emptyset = + \infty$. The $d$-height of a complex number $\alpha$, for some integer $d\geq 1$, is defined as
\begin{multline*}
H_d(\alpha)= \min \{ \max \{|c_0|,\dots, |c_d|\}, c_0, \dots, c_d \in \Z \mbox{ coprime,} \\ \mbox{not all zero and } c_0\alpha^d+\dots+c_d=0 \}.
\end{multline*}
For an $N$-tuple $(\alpha_1, \dots , \alpha_N) $, we set $H_d(\alpha_1, \dots , \alpha_N)= \max \{ H_d(\alpha_j) \}$.
Note that an $N$-tuple $(\alpha_1, \dots , \alpha_N) $ has finite $d$-height if and only if all the entries are algebraic numbers of degree at most $d$ over $\Q$. Moreover, if $\alpha \in \Q$ then $H_d(\alpha)=H(\alpha)$ for all $d$.

By Lemma 1.6.7 of \cite{BombGub} one can see that, if $\alpha$ is an algebraic number of degree at most $d$, then
\begin{equation} \label{heightequiv}
H_d(\alpha)\leq 2^d H(\alpha)^d.
\end{equation}

Let $\alpha$ be an imaginary quadratic number with minimal polynomial $aX^2+bX+c\in \Z[X]$. Then, we have
\begin{equation}\label{absv}
|\alpha|=\left|\frac{-b\pm \sqrt{b^2-4ac}}{2a} \right|\ll H_2(\alpha),
\end{equation}
\begin{equation}\label{H2re}
H_2(\text{Re}(\alpha)) \leq \max \{ |b|,|2a| \} \leq 2 H_2(\alpha),
\end{equation}
and 
\begin{equation}\label{H2im}
H_2(\text{Im}(\alpha))\leq \max \{ |b^2-4ac|,|4a^2| \} \ll  H_2(\alpha)^2.
\end{equation}

We call $A$ the quasi-projective variety in $Y(2)\times (\mathbb{P}^2)^n $ with coordinates $$(\lambda,[X_1,Y_1, Z_1],\dots ,[X_n,Y_n, Z_n ])$$ and defined by the $n$ equations
$$
Y_i^2Z_i=X_i(X_i-Z_i)(X_i-\lambda Z_i),
$$
for $i=1, \dots ,n$. We set $P_i=[X_i,Y_i,Z_i]$ and we have a curve $\cC\subseteq A$ defined over a number field $k$ such that $\lambda$ is non-constant. Then, if $\bo{c}_0$ is an algebraic point of $\cC$, using standard properties of heights we have that 
\begin{equation} \label{ineqheights}
h(P_i(\bo{c}_0)) \ll h(\lambda (\bo{c}_0))+1,
\end{equation}
for all $i=1, \dots, n$. Moreover, we have
\begin{equation}\label{ineqdeg}
[k(\bo{c}_0):k]\ll [k(\lambda(\bo{c}_0)):k].
\end{equation}

\subsection{Uniformisation}

We want to define a uniformisation map for $A$. For more details we refer to Chapter VII of \cite{Fordautom}.

It is well known that an elliptic curve over the complex numbers is analytically isomorphic to a complex torus $\C / \Lambda_{\tau}$ where $\tau$ is an element of the complex upper-half plane $\mathbb{H}$ and $\Lambda_\tau$ is the lattice generated by 1 and $\tau$. Moreover, let 
$$
\mc{L}_\tau =\lg z \in \C: z=x+\tau y , \mbox{ for }x, y \in [0,1) \rg,
$$
be a fundamental domain for such lattice.

The well-known Weierstrass $\wp$-function $\wp(z,1,\tau)=\wp(z,\tau)$ is a $\Lambda_{\tau}$-periodic function satisfying a differential equation of the form
\begin{equation}\label{weieq}
(\wp(z,\tau)')^2= 4 \wp(z,\tau)^3-g_2(\tau) \wp(z,\tau) - g_3(\tau),
\end{equation}
where $\wp(z,\tau)'=d\wp(z,\tau)/dz$.
Consider the values of the $\wp $-function at the half periods
$$
e_1(\tau)=\wp\left( \frac{1}{2}, \tau\right), \quad  e_2(\tau)=\wp\left( \frac{1+\tau}{2}, \tau\right), \quad e_3(\tau)=\wp\left( \frac{\tau}{2}, \tau\right).
$$
These are the zeroes of the cubic polynomial on the right hand side of \eqref{weieq}, i.e.,
\begin{equation}\label{weieq2}
(\wp(z,\tau)')^2= 4 (\wp(z,\tau)-e_1(\tau))  (\wp(z,\tau)-e_2(\tau))  (\wp(z,\tau)-e_3(\tau)).
\end{equation}
Note that the $e_i(\tau)$ are distinct and $e_3(\tau)-e_1(\tau)$ has a regular square root for all $\tau \in \mathbb{H}$.
Therefore, we can define
$$
\xi(z,\tau)=\frac{\wp(z,\tau)-e_1(\tau)}{e_3(\tau)-e_1(\tau)},
$$
and
$$
\eta(z,\tau)=\frac{\wp(z,\tau)'}{2(e_3(\tau)-e_1(\tau))^{3/2}}.
$$
By \eqref{weieq2} we have
$$
\eta(z,\tau)^2=\xi(z,\tau) (\xi(z,\tau)-1)(\xi(z,\tau)-\lambda(\tau)),
$$
where
$$
\lambda(\tau)=\frac{e_2(\tau)-e_1(\tau)}{e_3(\tau)-e_1(\tau)}.
$$
The map $(z,\tau)\mapsto (\lambda(\tau),P(z,\tau))$, where
\begin{equation*}
P(z,\tau)=\lg 
\begin{array}{ll}
[\xi(z,\tau),\eta(z,\tau),1], & \text{ if } z \not\in \Lambda_\tau ,\\
\left[ 0,1,0\right], & \text{ otherwise,}
\end{array} \right.
\end{equation*}
gives a parameterisation of the Legendre family.
Define
\begin{equation}\label{defpi}
\begin{array}{lll}
\pi: &   \mathbb{H} \times \C^n  &\rightarrow A\\
& ( \tau, z_1 , \dots , z_n) & \mapsto (\lambda(\tau), P(z_1,\tau), \dots , P(z_n,\tau))
\end{array}
\end{equation}
We would like to find a subset of $ \mathbb{H} \times \C^n $ over which it is possible to define a univalued inverse function of $\pi$.

By Chapter VII of \cite{Fordautom}, there exists a finite index subgroup $\Gamma$ of $\text{SL}_2(\Z)$ such that $\lambda(\gamma \tau)=\lambda(\tau)$ for all $\gamma \in \Gamma$. As a fundamental domain for the action of $\Gamma$ on $\mathbb{H}$ one can take the union of six suitably chosen fundamental domains for the action of $\text{SL}_2(\Z)$ (see Fig.~48 and 49 on p.~161 of \cite{Fordautom}). We call this set $\mc{B}$.

Now we set 
$$
\mc{F}_\mc{B}=\lg   (\tau, z_1, \dots, z_n): \tau \in \mc{B}, z_1, \dots , z_n \in \mc{L}_\tau   \rg.
$$ 
Then $\pi$ has a univalued inverse $A\rightarrow \mc{F}_\mc{B}$ and we define
\begin{equation}\label{defZ}
\mc{Z}=\pi^{-1} (\cC)\cap \mc{F}_\mc{B}.
\end{equation}

Finally, we consider a small open disc $D$ on $\cC$ and see $\tau ,z_1, \dots , z_n$ as holomorphic functions on $D$. The following is a consequence of Th\'eor\`eme 5 of \cite{Bertr}.

\begin{lemma}\label{lemBer}
Consider $\tau ,z_1, \dots , z_n$ as functions on $D$. If $1, \tau,z_1, \dots , z_n$ are $\Z$-linearly independent then $\tau,z_1, \dots , z_n$ are algebraically independent over $\C$.
\end{lemma}
\subsection{Complex Multiplication}

Suppose now that $E_{\lambda_0}$ has complex multiplication for some $\lambda_0$. Then, the associated $\tau_0 \in \mathcal{B}$ is a quadratic number with minimal polynomial $aX^2+bX+c$ and discriminant $D_0=b^2-4ac$. By Theorem 1 on p.~90 of \cite{Langell} we have End$(E_{\lambda_0})=\Oseen_{\lambda_0}=\Z[\rho]$, where $\rho=(D_0+\sqrt{D_0})/2$. Let $cl(\Oseen_{\lambda_0})$ be the class number of $\Oseen_{\lambda_0}$. Since the endomorphism $\rho$ has degree $(D_0^2-D_0)/4$, we have
\begin{equation}\label{heightrho}
h(\rho P) \ll |D_0|^2( h(P)+1),
\end{equation}
for any $P \in E_{\lambda_0}(\Qbar)$.

Now, by the general theory of complex multiplication we have that 
\begin{equation}\label{degcln}
[\Q(j_0):\Q]=cl(\Oseen_{\lambda_0}),
\end{equation}
where $j_0$ is the $j$-invariant of $E_{\lambda_0}$. Moreover, a theorem of Siegel in the form of Theorem 1.2 of \cite{Breuer} gives us the estimate 
\begin{equation}\label{BrSig}
|D_0|^{\frac{1}{2}-\epsilon} \ll_\epsilon  cl(\Oseen_{\lambda_0}) \ll_\epsilon |D_0|^{\frac{1}{2}+\epsilon}.
\end{equation}


We know that to $\lambda_0$ we can associate a unique $\tau_0 \in \mc{B}$ and a $\tau_0'$ in the usual fundamental domain for the action of SL$_2(\Z)$. These two have the same discriminant $D_0$. There is a finite set of elements of SL$_2(\Z)$ that sends any $\tau \in \mc{B}$ to the usual fundamental domain. Therefore, we have $H(\tau_0)\ll H(\tau_0')$.
If $a'X^2+b'X+c'$ is the minimal polynomial of $\tau_0'$, then $\tau'_0= (-b'\pm \sqrt{D_0})/(2a')$. Since $|$Re$(\tau_0')|\leq 1/2$ and Im$(\tau_0') \geq 1/2$, we have $|b'|\leq| a'| \leq |D_0|^{1/2}$. Therefore, by standard properties of heights we have
\begin{equation}\label{tauheightbound}
H(\tau_0)\ll H(\tau_0') \leq 2 H\left(\frac{b'}{2a'} \right) H\left(\frac{\sqrt{D_0}}{2a'}\right)  \ll   |D_0|^{\frac{3}{2}}.
\end{equation}

Recall the $q$-expansion of the $j$ invariant $j(\tau)=q^{-1}+744+196884q+\dots$ where $q=e^{2\pi i \tau}$. If $\tau$ is in the usual fundamental domain then $\text{Im}(\tau)\geq \sqrt{3} /2$ and therefore $$ \left| \log |j(\tau)|-2\pi \text{Im}(\tau)\right| \ll 1,$$
(see also equation (1) on p.~146 of \cite{Bilu13}). Now, let $\lambda_0, D_0, \tau_0'$ and $j_0$ be as above. We have that $\text{Im}(\tau_0')\leq |D_0|^{1/2}$. Now,  $j_0$ is an algebraic integer and therefore only the the archimedean places contribute to its height. Moreover, all conjugates of $j_0$ correspond to elliptic curves with complex multiplication with the same discriminant. Therefore, we have $$h(j_0)= \frac{1}{[\Q(j_0):\Q]} \sum \log^+ |j_0^\sigma| \ll |D_0|^{1/2}$$ and, since $j_0$ is a rational function of $\lambda_0$, we have
\begin{equation}\label{heightlamdisc}
h(\lambda_0)\ll |D_0|^{1/2}.
\end{equation} 

\section{O-minimality, definability and rational points}

For the basic properties and examples of o-minimal structures we refer to \cite{vandenDries1998} and \cite{DriesMiller}.

\begin{definition}
A \textit{structure} is a sequence $\mathcal{S}=\left( \mathcal{S}_N\right)$, $N\geq 1$, where each $\mathcal{S}_N$ is a collection of subsets of $\R^N$ such that, for each $N,M \geq 1$:
\begin{enumerate}
\item $\mathcal{S}_N$ is a boolean algebra (under the usual set-theoretic operations);
\item $\mathcal{S}_N$ contains every semialgebraic subset of $\R^N$;
\item if $A\in \mathcal{S}_N$ and $B\in \mathcal{S}_M$, then $A\times B \in \mathcal{S}_{N+M}$;
\item if $A \in \mathcal{S}_{N+M}$, then $\pi (A) \in \mathcal{S}_N$, where $\pi :\R^{N+M}\rightarrow \R^N$ is the projection onto the first $N$ coordinates.
\end{enumerate}
If $\mathcal{S}$ is a structure and, in addition,
\begin{enumerate}
\item[(5)] $\mathcal{S}_1$ consists of all finite union of open intervals and points
\end{enumerate}
then $\mathcal{S}$ is called an \textit{o-minimal structure}.
\end{definition}
Given a structure $\mathcal{S}$, we say that $S \subseteq \R^N$ is a \textit{definable set} if $S\in \mathcal{S}_N$. 
Let $S\subseteq \R^N$ and $f:S\rightarrow \R^M$ be a function. 
We call $f$ a \emph{definable function} if its graph $\lg (x,y) \in S\times \R^{M}:y=f(x) \rg$ is a definable set. It is not hard to see that images and preimages of definable sets via definable functions are still definable.

There are many examples of o-minimal structures. In this article we deal with sets definable in the structure $\R_{\text{an,exp}}$. The o-minimality of this structure was proved by van den Dries and Miller \cite{vandenDriesMiller}.

We now fix an o-minimal structure $\mathcal{S}$. 
We are going to use a result from \cite{HabPila14}. 

For a $Z \subseteq \R^{M+N}$, a positive integer $d$ and a positive real number $T$ we define
\begin{equation*}\label{def}
Z^{\sim}(d,T)=\lg (y,z)  \in  Z:  H_d(y) \leq T \rg.
\end{equation*}

By $\pi_1$ and $\pi_2$ we indicate the projections of $Z$ to the first $M$ and the last $N$ coordinates, respectively.

\begin{proposition}[\cite{HabPila14}, Corollary 7.2]\label{HabPila}
Let $Z\subseteq \R^{M+N}$ be a definable set.
For every $\epsilon>0$ there exists a constant $c=c(Z,d,\epsilon)$ with the following property. If $T\geq 1 $ and $|\pi_2\left( Z^{\sim}(d ,T) \right)|> c T^\epsilon$, then there exists a continuous definable function $\delta:[0,1] \rightarrow Z$ such that
\begin{enumerate}
\item the composition $\pi_1 \circ \delta : [0,1] \rightarrow \R^{M}$ is semi-algebraic and its restriction to $(0,1)$ is real analytic;
\item the composition $\pi_2 \circ \delta : [0,1] \rightarrow \R^N$ is non-constant.
\end{enumerate}
\end{proposition}

We now want to prove that the set $\mc{Z}$ defined in \eqref{defZ} is definable in $\R_{\text{an,exp}}$. From now on, by definable we mean definable in $\R_{\text{an,exp}}$ and complex sets and functions are said to be definable if they are as real objects considering their real and imaginary parts.

In \cite{PetStar}, Peterzil and Starchenko proved that, if $\mc{D}$ is the usual fundamental domain for the action of SL$_2(\Z)$ on $\mathbb{H}$, then $\wp (z,\tau)$ is a definable function when restricted to $\lg (\tau,z): \tau \in \mc{D}, z \in \mc{L}_\tau \rg$, and therefore definable if restricted to $\lg (\tau,z):\tau \in \gamma\mc{D}, z \in \mc{L}_\tau \rg$, where $\gamma \mc{D}$ is any fundamental domain for SL$_2(\Z)$. Since $\mc{B}$ is the union of six suitably chosen fundamental domains we have that $\wp (z,\tau)$ is also definable when restricted to $\lg (\tau,z):\tau \in \mc{B}, z \in \mc{L}_\tau \rg$. Therefore, the function $\pi$, defined in \eqref{defpi}, is definable when restricted to $\mc{F}_\mc{B}$. Finally, since $\cC $ is semialgebraic we can conclude that $\mc{Z}$ is definable.

\section{The main estimate}

Recall the definition of $\mc{Z}$ in \eqref{defZ}.
Define, for $T\geq 1$,
\begin{multline*}
\mc{Z}(T)= \lg  (\tau , z_1, \dots ,z_n) \in \mc{Z}: \sum a_j z_j \in \Z+\Z \tau, \mbox{ for some } (a_1, \dots ,a_n) \in \C^n\setminus \{ 0\},\right. \\ \left. \vphantom{\lg  (\tau , z_1, \dots ,z_n) \in \mc{Z}: \sum a_i z_i\in \Z+\Z \tau, \mbox{ for some } (a_1, \dots ,a_n) \in \C^n\setminus \{ 0\},\right.} \mbox{ with } H_2(\tau, a_1, \dots , a_n)\leq T   \rg .
\end{multline*}

Note that, even if for each value of $T$ this is a definable set, if we see $\mc{Z}(T)$ as family parameterized by $T$, it is not a definable family.

\begin{proposition}\label{mainest}
Under the hypotheses of Theorem \ref{mainthm}, for all $\epsilon>0$, we have $|\mc{Z}(T)|\ll_\epsilon T^\epsilon$, for all $T\geq 1$.
\end{proposition}

We indicate the imaginary unit by $I$ and define
\begin{multline*}
W=\lg (\alpha_{1}, \beta_{1}, \dots ,\alpha_n, \beta_n, \mu_1, \mu_2,u,v,x_1, y_1, \dots , x_n,y_n )\in \R^{4n+4}: \vphantom{\sum_{i=1}^n} \right.\\ \left. (u+vI,x_1+y_1 I,\dots , x_n+y_n I)\in \mc{Z}, \sum_{i=1}^n (\alpha_i+\beta_i I) (x_i+y_iI)= \mu_1 +\mu_2 (u+vI) \rg,
\end{multline*}
which is a definable set.

We want to apply Proposition \ref{HabPila} to $W$ with
$$
W^\sim (2,T)=\lg (\alpha_{1},  \dots , \beta_n, \mu_1, \mu_2,u,v,x_1, \dots ,y_n )\in W: H_2(\alpha_{1},  \dots , \beta_n, \mu_1, \mu_2,u,v)\leq T  \rg.
$$
We let $\pi_1$ and $\pi_2$ be the projections on the first $2n+4$ and the last $2n$ coordinates, respectively.

\begin{lemma}\label{lemest}
For all $\epsilon>0$, we have $|\pi_2\left( W^\sim (2,T)\right)|\ll_\epsilon T^\epsilon$, for all $T\geq 1$.
\end{lemma}

\begin{proof}
If $|\pi_2\left( W^\sim (2,T)\right)|> c T^\epsilon$, for some positive constant $c$, then Proposition \ref{HabPila} implies that there exists a continuous definable function $\delta:[0,1] \rightarrow W$ such that $\delta_1=\pi_1 \circ \delta : [0,1] \rightarrow \R^{2n+4}$ is semi-algebraic and $\delta_2=\pi_2 \circ \delta : [0,1] \rightarrow \R^{2n}$ is non-constant. This in turn implies that there exists a connected $J\subseteq [0,1]$ such that $\delta_1(J)$ is an algebraic curve segment and $\delta_2(J)$ has positive dimension.

We consider the coordinates $\alpha_{1},  \dots , \beta_n, \mu_1, \mu_2,u,v,x_1, \dots ,y_n$ as functions on $J$ and set $\tau=u+vI$ and $z_i=x_i+y_iI$. Moreover, we consider $$\mc{W}=(\tau(J),z_1(J),\dots ,z_n(J)) \subseteq \mc{Z}.$$ On $J$ the functions $\alpha_1, \dots , \beta_n,\mu_1, \mu_2, \tau$ generate a field of transcendence degree at most 1 over $\C$. Moreover, we have the relation
$$
\sum (\alpha_i+\beta_i I) z_i =\mu_1+\mu_2 \tau.
$$
Therefore, $\tau, z_1, \dots ,z_n$ are algebraically dependent on $J$.

Finally, we can consider $\tau, z_1, \dots ,z_n$ as functions on $\pi (\mc{W})$. They satisfy an algebraic relation which can be analytically continued to an open disc in $\cC$. By Lemma \ref{lemBer} this would imply that $1,\tau, z_1, \dots ,z_n$ are $\Z$-linearly dependent contradicting the hypothesis of generic independence of $P_1, \dots , P_n$.
\end{proof}

\begin{lemma}\label{lemomin}
There exists a positive constant $c'=c'(\mc{Z})$ such that for all $(z_1, \dots ,z_n)\in \C^n$ and $T$ there are at most $c$ values of $\tau \in \C$ with $(\tau , z_1, \dots ,z_n)\in \mc{Z}(T)$.
\end{lemma}

\begin{proof}
Consider the projection map $\varphi$ of $\mc{Z}$ to the last $n$ coordinates. By o-minimality, if $\varphi^{-1}(z_1, \dots , z_n)$ has dimension 0, then there is a uniform bound on its cardinality, which only depends on $\mc{Z}$. Therefore, we only need to prove that if $(\tau, z_1, \dots , z_n)\in \mc{Z}(T)$ for some $T$, then $\varphi^{-1}(z_1, \dots , z_n)$ has dimension 0. Suppose it has positive dimension and recall that the fixed $z_1, \dots , z_n$ are algebraically dependent over $\C$. As above, this would imply that the holomorphic functions $z_1, \dots , z_n$ would be algebraically dependent on some open disc in $\cC$, again contradicting Lemma \ref{lemBer}. 
\end{proof}

We are now in position to prove Proposition \ref{mainest}.

If $(\tau, z_1, \dots, z_n) \in \mc{Z}(T)$ then $\tau$ is imaginary quadratic and there are $a_1, \dots , a_n$ not all zero and each of degree at most 2 over $\Q$ and integers $a_{n+1},a_{n+2}$ with $ \sum a_i z_i = a_{n+1}+ a_{n+2} \tau$. 
Since $H_2(\tau, a_1, \dots , a_n)\leq T$ and $z_i \in \mc{L}_\tau$ and using \eqref{absv}, we have $|\sum a_i z_i|\leq \sum |a_i| |z_i| \ll T \max \{1, |\tau| \} \ll T^2$ and therefore we can suppose that $a_{n+1},a_{n+2}$ have absolute value $\ll T^2$.

By these considerations and by \eqref{H2re} and \eqref{H2im}, we have that
\begin{multline*}
\left(\text{Re} (a_1),\text{Im} (a_1), \dots ,\text{Re} (a_n),\text{Im} (a_n) ,a_{n+1},a_{n+2}, \right. \\ \left. \text{Re} (\tau),\text{Im} (\tau) ,  \text{Re} (z_1),\text{Im} (z_1), \dots ,\text{Re} (z_n),\text{Im} (z_n)  \right) \in W^\sim (2,\gamma T^2),
\end{multline*}
for some positive constant $\gamma$. By Lemma \ref{lemomin}, any point of $\pi_2\left (W^\sim (2,\gamma T^2)\right)$ is associated to at most $c'$ different elements of $\mc{Z}(T)$. Therefore, Lemma \ref{lemest} gives the claim of Proposition \ref{mainest}. 

\section{Proof of Theorem \ref{mainthm}}

In this section $\gamma_1, \gamma_2, \dots $  will be positive constants depending only on $\cC$. Recall that $\cC$ is defined over a number field $k$.

We call $\cC'$ the set of points $\bo{c} \in \cC(\C)$ such that $E_{\lambda (\bo{c})}$ has complex multiplication and there exists $(a_1, \dots, a_n) \in \text{End}(E_{\lambda (\bo{c})})^n \setminus \{0\}$ with
$$
a_1 P_1(\bo{c}) + \dots + a_n P_n (\bo{c})=O,
$$
i.e., the set we want to prove to be finite. Note that, if $E_{\lambda(\bo{c})}$ has complex multiplication, then $\lambda(\bo{c})$ is algebraic and therefore $\cC'$ consists of algebraic points of $\cC$.

Fix $\bo{c}_0 \in \cC'$, call $d_0$ its degree over $k$ and $D_0$ the discriminant of End$(E_{\lambda (\bo{c}_0)})$. Now, for all $\sigma \in \text{Gal}(\overline{k}/k)$, we have that all conjugates $\bo{c}_0^{\sigma}$ of $\bo{c}_0$ are in $ \cC'$. Actually, all End$(E_{\lambda (\bo{c}_0^{\sigma})})$ are isomorphic and 
\begin{equation}\label{eqrelsig}
a_1^{\sigma} P_1(\bo{c}_0^{\sigma}) + \dots + a_n^{\sigma} P_n (\bo{c}_0^{\sigma})=O,
\end{equation}
on $E_{\lambda (\bo{c}_0^{\sigma})}$. 

\begin{lemma}\label{lemgen}
For all $\bo{c}_0 \in \cC'$ there is $(a_1, \dots, a_n) \in \text{\emph{End}}(E_{\lambda (\bo{c}_0)})^n \setminus \{0\}$ satisfying \eqref{eqrelsig} with
$$
H_2(a_1, \dots ,a_n)\ll |D_0|^{\gamma_1} ,
$$
for some positive $\gamma_1$.
\end{lemma}

\begin{proof}
We fix a $\bo{c}_0 \in \cC'$ and set $\lambda_0=\lambda (\bo{c}_0)$, $K_0=k(\bo{c}_0, \sqrt{D_0})$, $\kappa_0=[K_0:\Q]$ and $P_i=P_i(\bo{c}_0 )$.
Recall that, by \eqref{ineqdeg}, \eqref{degcln} and \eqref{BrSig}, we have that $|D_0|^{1/3} \ll \kappa_0\ll |D_0|$.
 Moreover, let $\rho=\frac{D_0+\sqrt{D_0}}{2}$ so that End$(E_{\lambda_0})=\Z+\rho \Z$. 


Let $E$ be the elliptic curve with Weierstrass equation
$$
Y'^2=4X'^3-g_2 X' - g_3,
$$
where $g_2=\frac{4}{3}(\lambda_0^2-\lambda_0+1)$ and $g_3=\frac{4}{27} (\lambda_0-2)(\lambda_0+1)(2\lambda_0-1)$. Then $E_{\lambda_0}$ and $E$ are isomorphic via the map $\phi$ given by
$$
X'=X-\frac{1}{3}(\lambda_0+1), \quad  Y'=2Y,
$$
(see (3.7), p. 1683 of \cite{MasserZannier10}), and $E$ is defined over the same field $\Q(\lambda_0)$.

We set $Q_i=\phi (P_i)$ and $Q_i'=\phi(\rho P_i)$. Then, $Q_1, \dots , Q_n, Q'_1, \dots , Q'_n$ are $2n$ points in $E(K_0)$ such that $Q_i$ and $P_i$ have same N\'eron-Tate height and the same holds for $Q_i'$ and $\rho P_i$. Moreover, $Q_1, \dots , Q_n, Q'_1, \dots , Q'_n$ and $P_1, \dots , P_n, \rho P_1, \dots , \rho P_n$
 satisfy the same relation over $\Z$. 

Let $w=\max \{1,h(g_2),h( g_3)\}$. Using the work of Zimmer \cite{Zimmer} in the form of Lemma 3.4 of \cite{David95} and by \eqref{ineqheights} and \eqref{heightlamdisc} we have, for all $i$,
$$
\widehat{h}(Q_i)\leq h(Q_i) +\frac{3}{4} w + 5 \log(2)\ll h(P_i) +h(\lambda_0)+1 \ll h(\lambda_0)+1\ll |D_0|^{1/2}.
$$
Similarly, using \eqref{heightrho}, we have
$$
\widehat{h}(Q'_i)\ll  |D_0|^2 h(P_i)+h(\lambda_0)+1 \ll |D_0|^{5/2}.
$$

Suppose at least one $P_i$ has infinite order. We use a result of Masser. By Theorem E of \cite{Masser88} we can suppose that $Q_1, \dots , Q_n, Q'_1, \dots , Q'_n$ satisfy
$$
b_1 Q_1 + \dots +b_n Q_n+b_1' Q_1' + \dots +b_n' Q_n'=O
$$
on $E$, for some integers $b_1,\dots, b_n, b'_1, \dots, b'_n$, not all zero, with 
$$
\max \{|b_1| , \dots ,|b_n| ,|b_1'|, \dots ,| b_n'|\} \leq (2n)^{2n-1} \omega \left( \frac{q}{\eta} \right)^{\frac{1}{2}(2n-1)},
$$
where $\omega=|E_{\text{tors}}(K_0)|$, $\eta=\inf \widehat{h}(Q)$ for $Q\in E(K_0) \setminus E_{\text{tors}}(K_0)$ and $q\geq \eta$ is an upper bound for $\widehat{h}(Q_i)$ and $\widehat{h}(Q_i')$. By our previous considerations we can take $q \ll |D_0|^{5/2}$. We need to bound $\omega$ and $\eta$.

For the first we use a result of David \cite{David97}. By his Th\'eor\`eme 1.2(i) we have, after having chosen an archimedean $v$ and noting that $h_v(E)\geq \frac{\sqrt{3}}{2}$,
$$
\omega \ll \kappa_0 t +\kappa_0 \log \kappa_0,
$$
where $t=\max \{1, h(j_E) \}$. Since $E$ and $E_{\lambda_0}$ are isomorphic they have the same $j$-invariant and therefore, by \eqref{heightlamdisc}, we can take $t\ll |D_0|^{1/2}$. Therefore we have $\omega \ll |D_0|^{\gamma_2}$.

The lower bound on $\eta$ relies on a result of Masser. By Corollary 1 of \cite{Masser89} we have
$$
\eta \gg \kappa_0^{-\gamma_3} w^{-\gamma_4},
$$
where recall that $w=\max \{1,h(g_2),h( g_3)\} $. Actually, in Masser's bound a constant depending on $\kappa_0$ appears in the denominator but going through the proof one can see that it depends polynomially on $\kappa_0$, as noted by David on p.~109 of \cite{David97}. We have $w \ll h(\lambda_0)+1$ which is in turn $\ll |D_0|^{1/2}$. Therefore, we have $\eta \gg  |D_0|^{\gamma_5}$.
Combining the bounds on $q$, $\omega$ and $\eta$ we can suppose  
$$
\max \{|b_1| , \dots ,|b_n| ,|b_1'|, \dots ,| b_n'|\} \ll |D_0|^{\gamma_6}.
$$
In case all $P_i$ are torsion we can get the same estimate using only the bound on $\omega$.

Finally, by \eqref{heightequiv} we have $H_2(b_i+\rho b_i')\leq 4 H(b_i+\rho b_i')^2$ for all $i$, and using standard properties of heights we have the claim.
\end{proof}

Now, recall the definition \eqref{defpi} of the map $\pi$ and consider $\pi^{-1}(\bo{c}_0^\sigma)\cap \mc{F}_{\mc{B}}$ which consists of one point $(\tau_0^\sigma, z_1^\sigma, \dots , z_n^\sigma)$ belonging to $\mc{Z}$. We remark that, for varying $\sigma$, the points $(\tau_0^\sigma, z_1^\sigma, \dots , z_n^\sigma)$ are not Galois conjugates of each other.
We have
$$
a_1 z_1^\sigma + \dots +a_n z_n^\sigma \in \Z + \tau_0^\sigma \Z,
$$
with $H_2(a_1, \dots , a_n)\ll |D_0|^{\gamma_1}$.

Recall that, by \eqref{heightequiv} and \eqref{tauheightbound}, we have $H_2(\tau_0^\sigma)\ll |D_0|^3$. This implies $\pi^{-1}(\bo{c}_0^\sigma) \in \mc{Z}(|D_0|^{\gamma_7})$, for some positive $\gamma_7$. Now, by \eqref{BrSig} and \eqref{degcln}, we have $d_0\gg |D_0|^{1/3}$.

Thus, there are $\gg |D_0|^{\frac{1}{3}} $ different $(\tau_0^\sigma, z_1^\sigma, \dots , z_n^\sigma)$ in $\mc{Z}(|D_0|^{\gamma_7})$. Applying Proposition \ref{mainest} with $\epsilon= 1/(4 \gamma_7) $ we get a contradiction if $|D_0|$ is large enough.
Therefore, $|D_0|$ is bounded giving us the claim of Theorem \ref{mainthm}.

\section*{Acknowledgments}

We would like to thank Gareth Jones and Harry Schmidt for many useful discussions and Philipp Habegger for his comments and suggestions.

\bibliographystyle{plain}
\bibliography{bibliography.bib}

\end{document}